\documentclass[a4paper, 11pt]{amsart}

\usepackage{amsmath, amsfonts, amssymb}
\usepackage{amsthm}
\usepackage{esint}
\usepackage{tikz}
\usetikzlibrary{matrix}
\usepackage{cite}
\usepackage{array}
\usepackage{tikz-cd}
\usepackage [english]{babel}
\usepackage [autostyle, english = american]{csquotes}
\MakeOuterQuote{"}

\usepackage{amsaddr}
\usepackage{enumerate}
\usepackage{pgfplots}

\title{Around $L^1$ (un)boundedness \\ of Bergman and Szeg\"o projections}

\author{Gian Maria Dall'Ara}
\address{Istituto Nazionale di Alta Matematica "F. Severi"\\ Unità di Ricerca SNS Pisa}
\email{dallara@altamatematica.it}
\thanks{}

\date{\today}

\newcommand{\C}{\mathbb{C}}
\newcommand{\R}{\mathbb{R}}

\newcommand{\Z}{\mathbb{Z}}
\newcommand{\N}{\mathbb{N}}

\newtheorem{thm}{Theorem}[section]
\newtheorem{dfn}[thm]{Definition}
\newtheorem{ex}[thm]{Example}
\newtheorem*{ex2}{Setting A}
\newtheorem*{ex3}{Setting B}
\newtheorem{prop}[thm]{Proposition}
\newtheorem{lem}[thm]{Lemma}
\newtheorem{cor}[thm]{Corollary}

\begin{document}

\maketitle

\begin{abstract}
We consider the problem of $L^1$ (un)boundedness for a wide class of orthogonal projections, including Bergman projections on domains in complex manifolds and Szeg\"o projections on abstract CR manifolds. 	
\end{abstract}

\tableofcontents

\section{Introduction}

A variety of interesting operators in complex and harmonic analysis arise as orthogonal projections on $L^2$ spaces, examples being Bergman and Szeg\"o projections. As such, they are automatically bounded on $L^2$, while understanding their behavior on different function spaces, e.g., $L^p$ spaces, may be a difficult problem. Let us formulate precisely this problem in an abstract setting. 

One is given a measure space $(X,\mu)$ and an orthogonal projection \[
\Pi:L^2(X,\mu)\longrightarrow H,
\] where $H$ is a closed subspace of $L^2(X,\mu)$, and wants to know for which values of $p\in [1,+\infty]$ the a priori bound \[
(B_p)\qquad ||\Pi(f)||_{L^p(X,\mu)}\leq C||f||_{L^p(X,\mu)}\qquad \forall f\in L^2(X,\mu)\cap L^p(X,\mu)
\] holds, for some finite positive constant $C$. 
The set \begin{equation}\label{I}
\mathcal{I}(\Pi):=\left\{\frac{1}{p}\colon \ (B_p) \text{ holds}\right\}
\end{equation}
is an interval containing $\frac{1}{p}=\frac{1}{2}$, and it is symmetric with respect to this value, i.e., $\alpha\in \mathcal{I}(\Pi)$ if and only if $1-\alpha \in \mathcal{I}(\Pi)$. These properties are immediate, and well-known, consequences of self-adjointness of $\Pi$ on $L^2(X, \mu)$ and interpolation theory (see, e.g., \cite{bergh_lofstrom}). \newline 
In view of this, given a projection $\Pi$, a natural question is \emph{whether $\mathcal{I}(\Pi)=[0,1]$ or, equivalently, whether $\Pi$ admits a bounded extension to $L^1(X,\mu)$.} Our goal is to investigate this question for Bergman, Szeg\"o, and other classes of projections. In the next two sections we present our main results about Bergman and Szeg\"o projections, deferring to Section \ref{overview_sec} a more general discussion.

\subsection{The $L^1$ (un)boundedness problem for Bergman projections}

We recall the definition of Bergman projection of a domain in complex Euclidean space.

\begin{dfn}[Bergman projections on domains in $\C^n$] Let $\lambda$ be the standard $2n$-dimensional Lebesgue measure on $\C^n$. If $D\subset\C^n$ is a domain, the associated Bergman projection $B_D$ is the orthogonal projection on $L^2(D)=L^2(D,\lambda)$ with range the Bergman space $A^2(D)$ of Lebesgue square-integrable holomorphic functions.\end{dfn}

The set $\mathcal{I}(B_D)$ defined in \eqref{I} has been determined to be $(0,1)$ for various classes of bounded and smoothly bounded domains, among which: strongly pseudoconvex domains \cite{phong_stein, lanzani_stein}, pseudoconvex domains of finite-type in $\C^2$ \cite{nagel_rosay_stein_wainger}, convex domains of finite-type in $\C^n$ \cite{mcneal_stein}, pseudoconvex domains with diagonalizable Levi form \cite{charpentier_dupain}. In all these cases, the $L^1$ unboundedness of $B_D$ is an obvious corollary of a very precise description of the corresponding Bergman kernels in terms of an appropriate pseudometric (for which we refer to \cite{mcneal_sing_int}). $L^1$ unboundedness of Bergman projections seems to be known in virtually any other class of domains considered in the literature, even when the much harder question of completely describing $\mathcal{I}(B_D)$ is still open: see \cite{barrett_worm, barrett_sahutoglu, edholm_mcneal} and the survey paper \cite{zeytuncu} for examples.

Considering the amount of evidence reported above, one is led to conjecture that $1\notin \mathcal{I}(B_D)$, at least for every bounded and smoothly bounded pseudoconvex domain in $\C^n$. We are not aware of any statement to this effect in the literature. Our first result proves the unboundedness without any pseudoconvexity assumption. 

\begin{thm}\label{intr_bergman_thm}
The Bergman projection $B_D$ of a bounded and smoothly bounded domain $D\subset \C^n$ is unbounded on $L^1$. 
\end{thm}

In Section \ref{bergman_sec} the reader can find a more general $L^1$ unboundedness statement for Bergman projections on precompact domains in complex manifolds satisfying appropriate assumptions, from which Theorem \ref{intr_bergman_thm} follows easily. 

It is natural to ask whether the same result holds on \emph{every bounded domain in $\C^n$}. We were not able to answer this question (even if the hypothesis of Theorem \ref{intr_bergman_thm} can be relaxed a bit, see Section \ref{bergman_sec}). Notice that both the boundedness assumption and some restriction on the ambient manifold are necessary, because Bergman spaces on unbounded domains can be trivial (e.g., when $D=\C^n$, by the $L^2$ version of Liouville theorem), and Bergman spaces on compact manifolds contain only locally constant functions, and in both cases the corresponding Bergman projections are trivially $L^1$ bounded.

\subsection{The $L^1$ (un)boundedness problem for Szeg\"o projections}

The second class of operators we want to consider are Szeg\"o projections, which may be defined in great generality as follows.

\begin{dfn}[Szeg\"o projections on abstract CR manifolds]\label{szego_ex} Let $X$ be a CR manifold of hypersurface type equipped with a smooth positive measure $\mu$. The corresponding Szeg\"o projection, denoted by $\mathcal{S}_{X,\mu}$, is the orthogonal projection on $L^2(X,\mu)$ with range the space of $\mu$-square-integrable CR functions (see Section \ref{involutive_sec} and the references there for a more detailed discussion of this definition). \end{dfn}

As in the case of Bergman projections, the $L^1$ unboundedness of a variety of Szeg\"o projections has been established as a corollary of a much deeper analysis of the relevant Szeg\"o kernels (e.g., when $X$ is the boundary of a pseudoconvex domain of one of the kinds we have seen when discussing Bergman projections \cite{phong_stein,nagel_rosay_stein_wainger, mcneal_stein, charpentier_dupain, lanzani_stein_szego}, but see also \cite{christ} for a class of abstract pseudoconvex compact CR manifolds of dimension 3). 

We have an analogue of Theorem 	\ref{intr_bergman_thm} for Szeg\"o projections. 

\begin{thm}\label{intr_szego_thm}
If $X$ is the boundary of a bounded and smoothly bounded domain $D\subset\C^n$, equipped with the induced CR structure and an arbitrary smooth positive measure $\mu$, then the corresponding Szeg\"o projection $\mathcal{S}_{X,\mu}$ is unbounded on $L^1$. 
\end{thm}

Both Theorem \ref{intr_bergman_thm} and Theorem \ref{intr_szego_thm} are probably unsurprising, and their proofs are indeed quite simple. However, things get more interesting if one considers the $L^1$ (un)boundedness problem for abstract CR manifolds. 

In fact, in the generality of Definition \ref{szego_ex}, one cannot expect to always have $1\notin\mathcal{I}(\mathcal{S}_{X,\mu})$. A first reason is the existence of CR structures not admitting any nonconstant $L^2$ CR function, as in the following example (but see also \cite{cordaro}). 

\begin{ex}[Barrett tori, \cite{barrett}]\label{barrett_ex} Let $X=\mathbb{T}^3:=(\R/(2\pi \Z))^3$ be the three-dimensional torus, endowed with the CR structure generated by \[
	L=\alpha(x_3)\frac{\partial}{\partial x_1}+\beta(x_3)\frac{\partial}{\partial x_2}+\frac{\partial}{\partial x_3},
	\] 
	where $(x_1,x_2,x_3)$ are angular coordinates and $\alpha, \beta\in C^\infty(\mathbb{T},\C)$. For $L$ to define a genuine CR structure we need $L$ and $\overline{L}$ to be $\C$-linearly independent, that is, $\Im(\alpha)$ and $\Im(\beta)$ to have no common zeros. The existence of a nonconstant $L^2$ CR function with respect to the above structure amounts to the existence of a nonzero $(m_1,m_2)\in \Z^2$ such that
	\[
	m_1\int_{\mathbb{T}}\alpha(x_3)dx_3+m_2\int_{\mathbb{T}}\beta(x_3)dx_3\in 2\pi \Z.
	\]
To verify this claim, one may apply $\overline{L}$ term by term to the Fourier expansion in $(x_1,x_2)$ of an arbitrary $f\in L^2(\mathbb{T}^3)$ (for details see \cite[p. 890]{barrett}). It is clear that for generic smooth, or even real analytic, $\alpha, \beta$ the above condition fails for every nonzero $(m_1,m_2)\in\Z^2$, that is, there are no nonconstant $L^2$ CR functions on the corresponding Barrett torus.\end{ex}
CR manifolds not admitting nonconstant $L^2$ CR functions do not exhaust the counterexamples to a hypothetical $L^1$ unboundedness result for general Szeg\"o projections, as the next example shows. 

\begin{ex}[Levi-flat CR structures]\label{levi_flat_ex} Let $X:=\mathbb{T}\times Y$, where $Y$ is any compact connected complex manifold, equipped with a smooth positive measure $\mu$ and the CR structure induced by the complex structure on $Y$ (i.e., the CR structure generated by $\frac{\partial}{\partial z_j}$, where $z_1,\ldots,z_n$ are local holomorphic coordinates on $Y$). A function $f$ is CR with respect to this structure if and only if $f(\theta, \cdot)$ is holomorphic, that is, constant, for every $\theta\in \mathbb{T}$. If we assume for simplicity that $\mu$ is the product of a smooth measure on $\mathbb{T}$ and a probability measure $\nu$ on $Y$, it is easy to see that \[
	\mathcal{S}_{X,\mu}f(\theta, z):=\int_Yf(\theta, \cdot)d\nu\qquad f\in L^2(X, \mu),
	\]
which is trivially bounded on $L^1(X,\mu)$.

The CR manifolds just described are very special instances of Levi flat CR structures (see Section \ref{exceptional_sec} and the references there for the definition). With a bit more effort, one may prove that Szeg\"o projections are $L^1$ bounded on compact Levi flat CR manifolds such that every leaf of the Levi foliation is compact. We omit the details.
\end{ex}

We are able to prove that, in the compact real analytic category, lack of nonconstant CR functions (Example \ref{barrett_ex}) and Levi-flatness (Example \ref{levi_flat_ex}) are essentially the only two sources of $L^1$ boundedness of Szeg\"o projections. The precise statement is as follows.

\begin{thm}\label{intr_CR_thm}
Let $X$ be a compact, connected, and real analytic CR manifold of hypersurface type, equipped with a smooth positive measure $\mu$. Then one of the following three conditions must hold: 
\begin{enumerate}
	\item[(a)] the Szeg\"o projection does not admit a bounded extension to $L^1(X,\mu)$;
	\item[(b)] $X$ admits no nonconstant real-analytic CR functions; 
	\item[(c)] $X$ is Levi-flat and every leaf of the Levi foliation is compact. 
\end{enumerate}
\end{thm}

A few remarks may be of help in clarifying the theorem. \begin{enumerate}
\item By real analytic CR manifold, we mean that $X$ has the structure of real analytic manifold and that the CR bundle is locally generated by real analytic sections. 
\item The core of the proof consists in showing that if a CR manifold as in the statement admits at least a nonconstant real analytic CR function and its Szeg\"o projection is bounded on $L^1$, then it must be Levi flat, thus providing an unexpected implication between a functional analytic property and a rather strong geometric one. 
\item This core implication is achieved appealing to an argument based on abstract functional analysis and the theory of wave front sets, valid in much greater generality, combined with a more refined analysis exploiting real analyticity. It would be interesting to know whether Theorem \ref{intr_CR_thm} holds in the smooth category. 
\item One may wonder whether condition (b) may be replaced by the stronger \begin{enumerate}
	\item[\emph{(b')}] \emph{$X$ admits no nonconstant $L^2$ CR function.}
	\end{enumerate}
This seems to be a difficult question, related to the hard problem of deciding if any $L^2$ CR function on $X$ can be globally $L^2$ approximated by real analytic CR functions (local approximation properties are better understood, see \cite[Chapter II]{ber_cord_hou}). We observe that conditions \emph{(b)} and \emph{(b')} are equivalent for Barrett tori (as may be easily deduced from our discussion in Example \ref{barrett_ex} or the one in the paper \cite{barrett} itself) and for the CR structures considered in \cite{cordaro} (we thank P. Cordaro for this observation). \end{enumerate}

\subsection{General overview of the paper}\label{overview_sec}

Our analysis is based on the observation that both Bergman and Szeg\"o projections may be embedded in the large family of $L^2$ projections onto spaces of solutions of involutive structures, as described in Section \ref{settings_sec}. The key property shared by any such projection $\Pi:L^2(X,\mu)\longrightarrow H$ is that the range $H$ consists of $L^2$ solutions of a system of first order homogeneous PDEs, and therefore it is essentially closed under pointwise multiplication (by Leibniz rule). While this cannot be literally true, because products of $L^2$ functions need not be $L^2$ and, more subtly, products of nonsmooth solutions need not be solutions without further assumptions (see, e.g., \cite[Corollary II.3.4]{ber_cord_hou}), it is enough for our purposes to notice that $H$ contains the subalgebra $A$ consisting of smooth solutions. This leads us to consider an even more general class of $L^2$ projections equipped with a subalgebra $A$ of their range $H$. This is the Abstract Setting discussed in Section \ref{gen_setting_sec}. The sequence of progressive generalizations we just described is subsumed in the following scheme: 
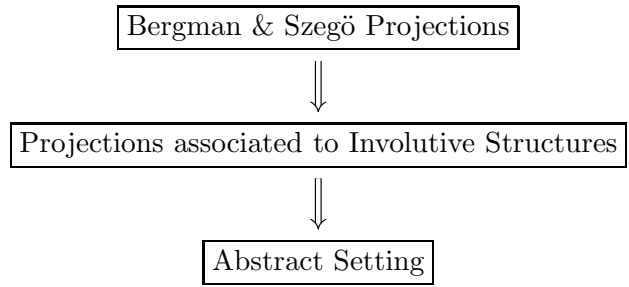
\begin{center}
	\begin{tikzcd}
		\boxed{\textrm{Bergman } \& \textrm{ Szeg\"o Projections} \arrow[d, Rightarrow]} \\
		\boxed{\textrm{Projections associated to Involutive Structures} \arrow[d, Rightarrow]} \\
		\boxed{\textrm{Abstract Setting}}
		\end{tikzcd}
	\end{center}

In the following we proceed logically, first defining the Abstract Setting in Section \ref{gen_setting_sec}, and then proving a general $L^1$ unboundedness result in Section \ref{abstract_thm_sec} (Theorem \ref{abstract_thm}). Its proof relies in a crucial way on Bishop's theorem about function algebras of holomorphic functions of one complex variable. We also show how to apply this general result to invariant projections on compact connected abelian groups (Theorem \ref{invariant_thm}), revealing a connection with the notion of Riesz set.

Next, in Section \ref{involutive_sec}, we introduce the class of $L^2$ projections associated to involutive structures on manifolds, including as special cases Bergman and Szeg\"o projections, and we deduce Theorem \ref{intr_bergman_thm} and Theorem \ref{intr_szego_thm} from the general results of Section \ref{abstract_thm_sec}. The reader interested in these theorems should keep in mind that their proofs rely on a more elementary version of the general $L^1$ unboundedness theorem valid in the Abstract Setting (namely, Theorem \ref{abstract_2_thm}). 

Finally, Section \ref{wf_sec} contains two lemmas about wave front sets needed in the proof of Theorem \ref{intr_CR_thm}, which occupies Section \ref{szego_sec}. 

Let us conclude by observing that, while we limited our discussion to the compact and scalar-valued setting, it should be possible to adapt our methods to more general situations, where the underlying spaces are allowed to be noncompact and the $L^2$ spaces consist of sections of a vector bundle. 

\subsection{Acknowledgements} The work presented in this paper was initiated while the author was a postdoc at the University of Vienna, it was developed while he was a Marie Curie Fellow at the University of Birmingham, and it was completed while he was Ricercatore Indam at Scuola Normale Superiore of Pisa.  This research was funded by the FWF-project P28154, by the European Union’s Horizon 2020 Research and Innovation Programme under the Marie Sk\l odowska-Curie Grant Agreement No. 841094, and by Istituto Nazionale di Alta Matematica "F. Severi". 

The author would like to thank Bernhard Lamel (University of Vienna) for various inspiring conversations on Szeg\"o projections on abstract CR manifolds, Paulo D. Cordaro (University of S\~{a}o Paulo) for offering his thoughts about CR manifolds not admitting nonconstant CR functions, and Fulvio Ricci (Scuola Normale Superiore, Pisa) for his comments on various aspects of this work. 

\section{The Abstract Setting}\label{gen_setting_sec}

Let $X$ be a \emph{compact connected Hausdorff} topological space, equipped with a \emph{finite Borel} measure $\mu$, which we assume to be \emph{regular} and of \emph{full support}, i.e., $\mu(V)>0$ for every nonempty open set $V\subseteq X$. The regularity assumption is automatically satisfied if $X$ has the additional property that each of its open sets is $\sigma$-compact, that is, it is a countable union of compact sets (see \cite[Theorem 2.18]{rudin}). 

Let $C(X)$ be the Banach algebra of continuous complex-valued functions on $X$, normed by the supremum norm, and let $\mathcal{M}(X)$ be the space of regular complex Borel measures on $X$, normed by total variation. By the Riesz Representation Theorem, $\mathcal{M}(X)$ is isometrically isomorphic to the dual of $C(X)$ via the pairing $(\mu, f)\mapsto \int_Xfd\mu$, and may be equipped with the corresponding weak-$\star$ topology. 

Since $\mu$ is finite, we have the following chain of continuous inclusions:\begin{equation}\label{inclusions}
C(X)\hookrightarrow L^2(X,\mu)\hookrightarrow L^1(X,\mu)\hookrightarrow \mathcal{M}(X),
\end{equation}
where the last one is the isometry that identifies $f\in L^1(X,\mu)$ with the $\mu$-absolutely continuous complex measure with density $f$.

Let $H$ be a \emph{closed linear subspace} of the Hilbert space $L^2(X,\mu)$, and let \[\Pi:L^2(X,\mu)\rightarrow H\] be the associated orthogonal projection. 

In the sequel an important rôle is played by $\mathcal{M}_0\subseteq\mathcal{M}(X)$, the \emph{weak-$\star$ closure} of $H$, when the latter is thought of as a subspace of $\mathcal{M}(X)$ as in \eqref{inclusions}. 

Finally, we assume to have a \emph{subalgebra with unit} \[A\subseteq C(X)\cap H,\] that is, a linear subspace of $C(X)\cap H$ containing the constant functions and closed under pointwise multiplication. Notice that we are not assuming that $A$ be closed in the sup-norm topology of $C(X)$. For the reader's convenience, we summarize the situation with a diagram.

\begin{center}
\begin{tikzcd}
	&L^2(X,\mu)\arrow{d}{\Pi}&\mathcal{M}(X)\\
	A\arrow[hook]{r}{}&H\arrow[hook]{r}{}&\mathcal{M}_0\arrow[hook]{u}{}
	\end{tikzcd}
\end{center}

The unlabelled arrows are the obvious inclusions.

\section{A general $L^1$ unboundedness result}\label{abstract_thm_sec}

Let $X$, $\mu$, $H$, $\Pi$, $\mathcal{M}_0$, and $A$ be as in Section \ref{gen_setting_sec}. We recall that we are interested in the following question: \emph{can $\Pi$ be extended to a bounded linear operator on $L^1(X,\mu)$?} 

Observe that, since $L^2(X,\mu)$ is dense in $L^1(X,\mu)$, $\Pi$ admits such an extension if and only if the following a priori inequality holds for some $C<+\infty$: \begin{equation}
	||\Pi(f)||_{L^1(X,\mu)}\leq C||f||_{L^1(X,\mu)}\qquad\forall f\in L^2(X,\mu).\label{L1_bound}
\end{equation}	

A general negative result is provided by the following theorem. 

\begin{thm}\label{abstract_thm}
Assume that: 
	\begin{itemize}
		\item[i)] $A$ is a nontrivial subalgebra, that is, there exists a nonconstant function $f_0\in A$;
		\item[ii)] every measure in $\mathcal{M}_0$ is absolutely continuous with respect to $\mu$.
	\end{itemize}	
	Then $\Pi$ cannot be extended to a bounded linear operator on $L^1(X,\mu)$. 
\end{thm}

The rest of this section is organized as follows: Section \ref{representing_sec} and Section \ref{bishop_mergelyan_sec} discuss various preliminaries, Section \ref{abstract_thm_proof_sec} contains the proof of Theorem \ref{abstract_thm}, Section \ref{variant_sec} presents a simpler variant of this result, and Section \ref{invariant_proj_sec} contains a first application to invariant projections on compact connected abelian groups.

\subsection{$L^1$ boundedness and representing measures}\label{representing_sec}

We show that if $\Pi$ admits a bounded extension to $L^1(X,\mu)$, then evaluation functionals on $A$ are represented by conjugates of measures in $\mathcal{M}_0$. 

\begin{lem}\label{representing_meas_lem}
	If $\Pi$ admits a bounded extension to $L^1(X,\mu)$, then for every $x\in X$ there exists $\sigma_x\in \mathcal{M}_0$ such that \begin{equation}
		\int_X fd\overline{\sigma_x} = f(x) \qquad \forall f\in C(X)\cap H.\label{representing_measure}
	\end{equation}
\end{lem}

\begin{proof}[Proof of Lemma \ref{representing_meas_lem}]
	Given an open neighborhood $U$ of $x$, choose a nonnegative continuous function $g_U$ supported on $U$ such that $\int g_Ud\mu=1$. The existence of such a function follows from Urysohn's Lemma and the fact that $\mu(U)>0$. Then \[
	\lim_{U\rightarrow x}\int fg_Ud\mu = f(x) \qquad \forall f\in C(X), 
	\]
	where the limit is in the sense of nets (if $X$ is first countable, one can of course work with a sequence of functions $g_k$ supported on a countable basis of neighborhoods of $x$). If $f\in C(X)\cap H$, since $\Pi$ is self-adjoint we have \begin{eqnarray*}
		f(x)&=&\lim\int f g_Ud\mu \\
		&=&\lim\int \Pi(f) \overline{g_U}d\mu \\
		&=&\lim\int f \overline{\Pi(g_U)}d\mu.
	\end{eqnarray*}
	By \eqref{L1_bound}, the net $\{\Pi(g_U)\mu\}_U$ is contained in a ball of $\mathcal{M}(X)$. By Banach--Alaoglu's Theorem, it has a weak-$\star$ convergent subnet. If $\sigma_x\in  \mathcal{M}(X)$ is its limit, then \eqref{representing_measure} holds. Obviously, $\sigma_x\in\mathcal{M}_0$.
\end{proof}

\subsection{Bishop's and Mergelyan's theorems}\label{bishop_mergelyan_sec}

The second tool we need for the proof of Theorem \ref{abstract_thm} is a consequence of two classical theorems of complex analysis, which we proceed to recall.

\begin{thm}[Mergelyan's Theorem]
	Let $K\subset\C$ be compact, and denote by $A(K)$ the algebra of continuous complex-valued functions on $K$ that are holomorphic in its interior. 
	
	If $\C\setminus K$ is connected, then every $f\in A(K)$ can be uniformly approximated by polynomials. 
\end{thm}
See \cite{gamelin}, Chapter II, Thm. 9.1 for a proof. We remark that here and in the sequel by polynomial we always mean a polynomial of a complex variable with complex coefficients (in other words, a holomorphic polynomial).

\begin{thm}[Bishop's Theorem]
	Let $K\subset\C$ be compact and assume that $\C\setminus K$ is connected. Then every boundary point $z\in bK$ is a peak point for the algebra $A(K)$, that is, for every $z\in bK$ there exists $F\in A(K)$ such that $F(z)=1$ and $|F(w)|<1$ for every $w\in K\setminus\{z\}$. 
\end{thm}
The function $F$ of the statement is called a \emph{peaking function}. See, e.g., the paper \cite{danielyan} for a proof and a discussion of this result.

Below, we will need to apply Mergelyan's and Bishop's theorems to a general compact set $K\subset\C$. One can proceed as follows. Let $U$ be the unique unbounded component of $\C\setminus K$, and put $K_1:=\C\setminus U$. Then $K_1$ is a compact set containing $K$, and can be informally described as the set obtained from $K$ by "filling in the holes". Its boundary $bK_1$ is sometimes called the \emph{outer boundary} of $K$. We denote it by $b_{\textrm{out}}K$. One may easily see that $b_{\textrm{out}}K\subset bK$. Since $K_1$ is compact and has connected complement, by Bishop's Theorem every $z\in b_{\textrm{out}}K$ is a peak point for $A(K_1)$. Let $F_z\in A(K_1)$ be a corresponding peaking function. By Mergelyan's Theorem, we can find a sequence $\{P_{z,k}\}_k$ of polynomials converging uniformly to $F_z$ in $K_1$, and hence in particular in $K$. We proved the following

\begin{cor}\label{bishop_mergelyan_cor}
	Let $K\subset\C$ be compact. Every point $z\in b_{\textrm{out}}K$ of the outer boundary of $K$ admits a peaking function $F\in A(K)$ that is a uniform limit of polynomials. 
\end{cor} 

Let us conclude this section with an elementary property of outer boundaries that will be invoked later. 

\begin{prop}\label{outer_prop} Let $K\subset\C$ be compact and connected. Then $\Re(K)=\Re(b_{\textrm{out}}K)$ and $\Im(K)=\Im(b_{\textrm{out}}K)$.
	
If, in addition, $K$ is not a singleton, then $b_{\textrm{out}}K$ is uncountable.
\end{prop}
Here $\Re(A):=\{\Re(z)\colon \ z\in A\}$, where $A\subseteq\C$ and $\Re(z)$ is the real part of $z$. Analogously, $\Im(A):=\{\Im(z)\colon \ z\in A\}$, where $\Im(z)$ is the imaginary part of $z$.
\begin{proof}
	The suprema and infima of any nonempty slice $K_x:=\{y\colon \ x+iy\in K\}$ (resp. $K^y:=\{x\colon \ x+iy\in K\}$) are elements of $\Re(b_{\textrm{out}}K)$ (resp. $\Im(b_{\textrm{out}}K)$). This proves the first statement. 
	
	If $K$ is not a singleton, then at least one of $\Re(K)$ and $\Im(K)$ is a nondegenerate closed interval. The second statement follows immediately. 
\end{proof}

\subsection{Proof of Theorem \ref{abstract_thm}}\label{abstract_thm_proof_sec}

We argue by contradiction, assuming that there exists $C<+\infty$ such that \eqref{L1_bound} holds. By Lemma \ref{representing_meas_lem}, for every $x\in X$ we can choose a measure $\sigma_x\in \mathcal{M}_0$ satisfying \eqref{representing_measure}. 

Fix any nonconstant $f_0\in A$, and let $K:=f_0(X)\subset\C$, which is compact, connected, and not a singleton. By Corollary \ref{bishop_mergelyan_cor}, for every $z\in b_{\textrm{out}}K$ there is $F_z\in A(K)$ such that $F_z(z)=1$ and $|F_z(w)|<1$ for every $w\in K\setminus \{z\}$, and there is a sequence $\{P_{z,k}\}_k$ of polynomials converging to $F_z$, uniformly in $K$. 

Since $A$ is an algebra with unit, $P_{z,k}(f_0)^N\in A$ for every $N\in\N$. As a uniform limit of elements of $A$, the continuous function $F_z(f_0)^N$ is an element of $H$. By \eqref{representing_measure}, \[
\int_X F_z(f_0)^Nd\overline{\sigma_x} = 1 \qquad \forall z\in b_{\textrm{out}}K,\quad \forall x\in f_0^{\leftarrow}\{z\}.
\]
By the dominated convergence theorem, letting $N$ tend to $\infty$ we conclude that $\sigma_x(f_0^{\leftarrow}\{z\})=1$ for every $x$ and $z$ as above. Since $\sigma_x$ is absolutely continuous with respect to $\mu$, we must have $\mu(f_0^{\leftarrow}\{z\})>0$ for every $z\in b_{\textrm{out}}K$. This is the desired contradiction, because $b_{\textrm{out}}K$ is uncountable, by Proposition \ref{outer_prop}, and the sets $f_0^{\leftarrow}\{z\}$ are pairwise disjoint. The proof is complete.

\subsection{A (more elementary) variant of Theorem \ref{abstract_thm}}\label{variant_sec}

Theorem \ref{abstract_thm} admits the following variant, that does not rely on Bishop's and Mergelyan's theorems.

\begin{thm}\label{abstract_2_thm} 
	Assume that there exists $x_0\in X$ such that: 
	\begin{itemize}
		\item[i')] $x_0$ is a peak point for $A$, that is, there exists $f_0\in A$ such that $f_0(x_0)=1$ and $|f_0(x)|<1$ for every $x\in X\setminus \{x_0\}$;
		\item[ii')] for every measure $\sigma\in\mathcal{M}_0$, we have $\sigma\{x_0\}=0$. 
	\end{itemize}	
	Then $\Pi$ cannot be extended to a bounded linear operator on $L^1(X,\mu)$. 
\end{thm}

We invite the reader to notice how i'), resp. ii'), is stronger, resp. weaker, than the corresponding hypothesis in Theorem \ref{abstract_thm}. 

\begin{proof}
	We argue by contradiction, assuming that \eqref{L1_bound} holds. Lemma \ref{representing_meas_lem} yields a measure $\sigma_{x_0}\in \mathcal{M}_0$ such that \begin{equation*}
		\int fd\overline{\sigma_{x_0}} = f(x_0) \qquad \forall f\in C(X)\cap H.
	\end{equation*}
	Applying this identity to $f_0^N\in A$ and letting $N$ tend to $\infty$, we get $\sigma_{x_0}\{x_0\}=1$, a contradiction. 
\end{proof}

\subsection{Invariant projections on compact connected abelian groups}\label{invariant_proj_sec}

We want to show that various invariant projections on compact connected abelian groups fit into the Abstract Setting of the previous sections.

 We refer to \cite[Chapter 4]{folland} for the foundations of Fourier analysis on locally compact abelian groups (l.c.a.g.'s). As is customary in this field, we assume that every topological group is Hausdorff. 

We recall that to every l.c.a.g. $G$, one may associate another l.c.a.g. $\widehat{G}$, called the \emph{dual group} and consisting of the unitary characters of $G$, namely, the continuous homomorphisms of $G$ into the circle group of complex numbers of modulus $1$, equipped with pointwise multiplication and the topology of compact convergence. A l.c.a.g. $G$ is compact if and only if $\widehat{G}$ is discrete (by \cite[Proposition 4.4]{folland}), and in this case $G$ is connected if and only if $\widehat{G}$ is torsion-free (by \cite[p. 99, Corollary 4]{morris}). Hence, by Pontrjagin duality, compact connected abelian groups are, up to isomorphism, the duals of discrete torsion-free abelian groups. The simplest example is the torus group $\left(\R/(2\pi\Z)\right)^n$ with dual $\Z^n$ ($n\in \N$), the setting of the classical theory of Fourier series. By the structure theory of finitely generated abelian groups \cite[Chapter I, Theorem 8.4]{lang}, the torus groups are the only compact connected abelian groups whose dual is finitely generated. 

Let $G$ be a compact connected abelian group, equipped with the normalized Haar measure $\mu$. Since the Fourier transform	\[
\mathcal{F}f(\xi):=\int_G f(x)\overline{\xi(x)}d\mu(x)
\] establishes a unitary isomorphism of $L^2(G,\mu)$ onto $\ell^2(\widehat{G})$, for any subset $\Gamma\subseteq \widehat{G}$ we have a corresponding closed subspace \[
H_\Gamma:=\{f\in L^2(G,\mu)\colon \ \mathcal{F}f(\xi)=0\quad \forall \xi\notin\Gamma\}
\] and the orthogonal projection $\Pi_\Gamma:L^2(G,\mu)\rightarrow H_\Gamma$. 

Now assume that $\Gamma$ is a submonoid of $\widehat{G}$, that is, $1\in \Gamma$ and $\xi_1+\xi_2\in \Gamma$ whenever $\xi_1,\xi_2\in \Gamma$ (here $1$ is the trivial character). \begin{flushleft}
	
\end{flushleft}One may easily verify that in this case \[
A_\Gamma:=\C-\textrm{span} \text{ of }\{\xi:G\rightarrow \C\colon \ \xi\in \Gamma\}
\] is a subalgebra with unit of $C(G)\cap H_\Gamma$. Thus, we are in the situation of the Abstract Setting of Section \ref{gen_setting_sec}. The weak-$\star$ closure $\mathcal{M}_0$ of $H_\Gamma$ in the space of regular complex measures $\mathcal{M}(G)$ consists of measures $\sigma$ such that \[
\mathcal{F}\sigma(\xi):=\int_G\overline{\xi(x)}d\mu(x)=0\qquad \forall \xi\notin \Gamma.
\]
The following definition is thus relevant. 

\begin{dfn}[Riesz set, \cite{meyer}]\label{riesz_dfn} A subset $\Gamma\subset \widehat{G}$ is said to be a Riesz set if every measure $\sigma\in \mathcal{M}(G)$ whose Fourier transform $\mathcal{F}\sigma$ vanishes off $\Gamma$ is absolutely continuous with respect to the Haar measure $\mu$. 
\end{dfn}

The classical F. and M. Riesz Theorem states that $\Gamma=\N$, thought of as a subset of the dual $\Z$ of $\R/(2\pi \Z)$, is a Riesz set. That any sector $\Gamma\subset\Z^n$ of angle less than $\pi$ is also a Riesz set is a theorem of Bochner \cite{bochner} (here the group is $\left(\R/(2\pi\Z)\right)^n$). We refer the reader to the literature on Riesz sets for more results of this kind. 

Now that all the ingredients are in place, we can specialize Theorem \ref{abstract_thm} 
to the present situation. 

\begin{thm}\label{invariant_thm}
Let $G$ be a compact connected abelian group, and let $\Gamma$ be a submonoid of its dual group $\widehat{G}$. If:\begin{itemize}
	\item[i)] $\Gamma$ is nontrivial, that is, $\Gamma\supsetneq\{1\}$, and
	\item[ii)] $\Gamma$ is a Riesz set,
	\end{itemize} then the projection $\Pi_\Gamma$ cannot be extended to a bounded linear operator on $L^1(G,\mu)$. 
\end{thm}

We leave to the interested reader the specialization of Theorem \ref{abstract_2_thm} to the present setting. 

Notice that if $G=\R/(2\pi \Z)$ and $\Gamma=\N$, then the operator $2\Pi_\Gamma-1$ is the classical Hilbert transform, whose $L^1$ unboundedness is well-known (it follows, e.g., from inspection of its singular integral kernel). The whole point of our specialization/digression to compact connected abelian groups is to show that our methods provide an abstract "cause" of $L^1$ unboundedness that does not depend on any computation or consideration specific to the torus group and, more importantly, that finds application in much greater generality, as we will show in the sequel. 

\section{Projection operators from involutive structures}\label{involutive_sec}

The goal of this section is to describe a wide variety of projection operators, including Bergman and Szeg\"o projections, in light of the Abstract Setting of Section 
\ref{gen_setting_sec}. In order to do that, we first need to recall the notion of involutive structure and distributional solution of an involutive structure. 

\subsection{Involutive structures} Let $X$ be a real smooth manifold. An \emph{involutive structure} on $X$ is a complex subbundle $E$ of the complexified tangent bundle $\C TX=\C\otimes TX$ satisfiying the following property, known as \emph{formal integrability}: if $Z$ and $W$ are smooth sections of $E$, then their commutator $[Z,W]$ is a section of $E$ too. As usual, we refer to the complex dimension of the fibers of $E$ as the rank of $E$. See \cite{ber_cord_hou} for various aspects of the analytic theory of involutive structures. Here we limit ourselves to recalling two important classes of examples (of which the second is actually a generalization of the first).

\begin{ex}[Complex structures] A complex structure is an involutive structure $E$ with the additional property that $\C T_xX=E_x\oplus \overline{E}_x$ for every $x\in X$, which in particular implies that the real dimension of $X$ is $2n$, where $n$ is the rank of $E$. A celebrated theorem of Newlander and Nirenberg states that this definition of complex structure is actually equivalent to the usual one in terms of a maximal complex atlas (see, e.g., \cite[p. 47]{ber_cord_hou}): in local holomorphic coordinates $z_1,\ldots, z_n$, $E=T_{1,0}X:=\textrm{span}\{\partial_{z_1}, \ldots, \partial_{z_n}\}$. One simply says that $X$ is a complex manifold. 
\end{ex}

\begin{ex}[CR structures] A CR structure on an $N$-dimensional real smooth manifold $X$ is a rank $n$ involutive structure $E$ with the additional property that $E_x\cap \overline{E}_x=0$ for every $x\in X$. In this case, one says that $(X,E)$ is a CR manifold of CR codimension $N-2n$. The real bundle $HX:=\Re(E)\subset TX$ (where $\Re$ denotes the real part) is called the horizontal bundle of the CR manifold, and has real codimension $N-2n$, justifying the terminology.
	
If the CR codimension is one, we say that $(X,E)$ is a CR manifold of hypersurface type. Notice that $X$ has odd real dimension $2n+1$ in this case. We recall that every real hypersurface $X$ of a complex manifold $Y$ is a CR manifold of hypersurface type in a natural way: the structure bundle is defined by $E:=T_{1,0}Y\cap \C TX$. A CR structure on a manifold $X$ is customarily denoted by $T_{1,0}X$, and one also writes $T_{0,1}X:=\overline{T_{1,0}X}$. 

There is a vast literature on this subject: see, e.g., \cite{boggess}, \cite{dragomir_tomassini}.
\end{ex}

We denote, as customary, by $\Gamma(U;E)$ the space of smooth sections of $E$ over the open set $U\subseteq X$. In particular, $\Gamma(U;\C TX)$ is the space of smooth complex vector fields on $U$. 

\subsection{Distributional solutions} 
Assume now that the real smooth manifold $X$ is endowed with a smooth positive measure $\mu$ (that is, a measure having a smooth positive density with respect to Lebesgue measure in local coordinates). \newline 
Let $U$ be an arbitrary open set of $X$. The "reference measure" $\mu$ allows us to think of distributions on $U$ as continuous linear functionals on $C^\infty_c(U)$ (see \cite[Section 6.3]{hormanderI}). We denote by $\mathcal{D}'(U)$ the space of distributions on $U$. The function $f\in L^1_{\textrm{loc}}(U)$ is identified with the distribution given by integration against $fd\mu$.

We recall how one may define the action of a smooth complex vector field $Z$ on a distribution $f\in \mathcal{D}'(U)$. First, one introduces the formal adjoint $Z^\dagger$ with respect to $\mu$, i.e., the first order partial differential operator $Z^\dagger$ uniquely defined by the identity \[
\int_X f Z^\dagger \varphi d\mu= \int_X\varphi Zfd\mu\qquad \forall f\in C^\infty(X), \varphi\in C_c^\infty(X). 
\]
A simple integration by parts shows that in local coordinates with respect to which $\mu=\widetilde\mu(x)dx_1\ldots dx_N$ and $Z=\sum_{k=1}^Na_k(x)\partial_{x_k}$, we have \[Z^\dagger=-\sum_{k=1}^Na_k(x)\partial_{x_k}-\widetilde\mu(x)^{-1}\sum_{k=1}^N\partial_{x_k}\left(\widetilde\mu(x)a_k(x)\right).\] 
Next, one declares that $Zf=g$, where $f,g\in \mathcal{D}'(U)$, if 
\[
\left\langle f, Z^\dagger \varphi\right\rangle = \left\langle g, \varphi\right\rangle \qquad \forall \varphi\in C_c^\infty(U),
\]
where the angle brackets denote the pairing between distributions and test functions. In particular, if $f,g\in L^1_{\textrm{loc}}(U)$, $Zf=g$ means that \[
\int_U f Z^\dagger \varphi d\mu= \int_Xg\varphi d\mu\qquad \forall \varphi\in C_c^\infty(U).
\]

Given an involutive structure $E$ and a space of distributions $\mathcal{F}(U)\subseteq \mathcal{D}'(U)$ on the open set $U\subseteq X$, we can now define the associated space of solutions as \[
\textrm{Sol}(E;\mathcal{F}(U)):=\{f\in \mathcal{F}(U)\colon\ Zf=0\quad \forall Z \in \Gamma(U;E)\}.\]

\begin{ex}
If $T_{1,0}X$ is a complex structure on $X$, then \[\mathcal{O}(U):=\textrm{Sol}(T_{0,1}X;C^1(U))=\textrm{Sol}(T_{0,1}X;\mathcal{D}'(U))\] is the space of holomorphic functions on $U$, while \[\mathcal{O}(U)\cap L^2(U,\mu)=\textrm{Sol}(T_{0,1}X;L^2(U,\mu))\] is the space of $\mu$-square integrable holomorphic functions on $U$, usually called \emph{weighted Bergman space} with "weight" $\mu$ in the literature. Recall that a distribution annihilated by every section of $T_{0,1}X$, that is, satisfying the Cauchy--Riemann equations, is automatically an ordinary smooth function by elliptic regularity.
\end{ex}

\begin{ex}
	If $T_{1,0}X$ is a CR structure on $X$, then $\textrm{Sol}(T_{0,1}X;\mathcal{D}'(U))$ is the space of CR distributions on $U$, $\textrm{Sol}(T_{0,1}X;C^1(U))$ is the space of CR functions of class $C^1$  on $U$, etc.
\end{ex}

\subsection{Projection operators}\label{settings_sec} We are now in a position to describe various instances of the Abstract Setting of Section 
\ref{gen_setting_sec}. 

\begin{ex2}\label{involutive_ex} Let $X$ be a compact connected real smooth manifold, equipped with a smooth positive measure $\mu$ and an involutive structure $E$. We put
\[H:=\mathrm{Sol}(E;L^2(X,\mu))\] and 
\[A:=\mathrm{Sol}(E;C^1(X)).\]
As in Section \ref{gen_setting_sec}, $\mathcal{M}_0\subseteq\mathcal{M}(X)$ is the weak-$\star$ closure of $H$, thought of as a subspace of $\mathcal{M}(X)$. Observe that:\begin{enumerate}
\item $H$ is a closed linear subspace of $L^2(X,\mu)$, and thus the orthogonal projection $\Pi:L^2(X,\mu)\rightarrow H$ is well-defined;
\item $\mathcal{M}_0\subseteq \mathrm{Sol}(E;\mathcal{M}(X))$;
\item $A$ is a subalgebra of $C(X)$ contained in $H$ and containing the constant functions.
	\end{enumerate}\end{ex2}

The first two observations follow immediately from the definitions and the fact that both $L^2$ and weak-$\star$ convergence imply convergence in the sense of distributions. The last one follows from the Leibniz rule.

Notice that, as a set, $H$ is independent of the choice of the measure, since the set $L^2(X,\mu)$ is independent of $\mu$, in virtue of the compactness of $X$. Nevertheless, the operator $\Pi$ depends both on the choice of the reference measure $\mu$ and of the involutive structure $E$. 

If $(X,T_{1,0}X)$ is a CR manifold of hypersurface type, the operator \[\mathcal{S}_{X,\mu}:L^2(X,\mu)\rightarrow \mathrm{CR}^2(X,\mu):=\textrm{Sol}(T_{0,1}X;L^2(U,\mu))\]
defined above is known as the \emph{Szeg\"o projection} on $X$ with respect to the weight $\mu$.

Specializing Theorem \ref{abstract_thm} and Theorem \ref{abstract_2_thm} to Setting A, we obtain the following result, where any element of $\mathrm{Sol}(E;\mathcal{D}'(X))$ is referred to, for simplicity, as a "solution". 

\begin{thm}\label{settingA_thm}
Let $X$, $\mu$, $E$, $H$, $\Pi$ be as in Setting A above. Assume that: 
\begin{itemize}
		\item[i)] there is at least a nonconstant $C^1$ solution;
		\item[ii)] every complex measure solution is absolutely continuous (with respect to $\mu$ or, equivalently, any other smooth positive measure).
	\end{itemize}	
	Then $\Pi$ cannot be extended to a bounded linear operator on $L^1(X,\mu)$. 
	
	The same conclusion holds if there is at least a $C^1$ solution that peaks at $x_0\in X$.
\end{thm}

\begin{proof}
The first part is just a translation of Theorem \ref{abstract_thm} to Setting A. The second part follows from Theorem \ref{abstract_2_thm} if we show that assumption ii') is automatically satisfied, i.e., that every complex measure solution is nonatomic. More precisely, if $\sigma$ is a regular complex measure and $L\sigma=0$ in the sense of distributions for some complex vector field $L$ that does not vanish at $x_0$, then $\sigma\{x_0\}=0$. This is an exercise in the theory of distributions, or a simple consequence of Lemma \ref{singular_measure_lem} and the arguments of Section \ref{wf_sec}. 
\end{proof}

We remark that both assumption i) and ii) are fairly nontrivial statements about the involutive structure. The first is a \emph{global existence} result, which fails, e.g., on compact connected complex manifolds, where the maximum principle forces any global holomorphic function to be constant, and may or may not hold on a compact CR manifold (as seen in the introduction). The second assumption is a \emph{regularity property} of the system of PDEs encoded by the involutive structure, which has been investigated in the literature under the name of \emph{F. and M. Riesz property}. See, e.g., \cite{brummelhuis, berhanu_hounie_2001,berhanu_hounie_2003, berhanu_hounie_2005}, and cf. Definition \ref{riesz_dfn}.

Setting A may be easily generalized to cover precompact domains of a manifold.

\begin{ex3}\label{domain_ex} Let $Y$ be a real smooth manifold, equipped with a smooth positive measure $\mu$ and an involutive structure $E$. Let $D\subseteq X$ be a precompact domain, i.e., a connected open set with compact closure $\overline{D}=:X$, so that the restriction $\mu_{|X}$ is finite. We put
	\[H:=\left\{f\in L^2(X,\mu)\colon \ f_{|D}\in \mathrm{Sol}(E;L^2(D,\mu))\right\}\] and 
	\[A:=\left\{f\in C(X)\colon \ f_{|D}\in \mathrm{Sol}(E;C^1(D))\right\}.\]
	As in Section \ref{gen_setting_sec}, $\mathcal{M}_0\subseteq\mathcal{M}(X)$ is the weak-$\star$ closure of $H$, thought of as a subspace of $\mathcal{M}(X)$. 
	
	Properties (1) and (3) of Setting A hold verbatim in the present case too. Notice that, on the contrary, property (2) does not make sense as it stands, because $\mathcal{M}(X)$ is not a space of distributions on an open set of $Y$, and hence $\mathrm{Sol}(E;\mathcal{M}(X))$ is not well-defined. The reader may easily prove the substitute property \begin{enumerate}
		\item[(2')] $\sigma\in \mathcal{M}_0\Longrightarrow \sigma_{|D}\in  \mathrm{Sol}(E;\mathcal{M}(D))$.
	\end{enumerate}
	\end{ex3}
The interested reader may formulate an analogue of Theorem \ref{settingA_thm} for Setting B.

We remark that under the very mild regularity assumption that $bD=X\setminus D$ has $\mu$-measure $0$ (a property that is independent of $\mu$) one can, and should, identify $L^2(X,\mu)=L^2(D,\mu)$, while keeping in mind that weak-$\star$ limits of sequences in $H\subseteq L^2(D,\mu)$ are taken with respect to the duality with $C(X)=C(\overline{D})$, and may thus have mass on the boundary $bD$. 

If $Y$ is a complex manifold equipped with a smooth positive measure $\mu$, and $D\subseteq Y$ is a precompact domain such that $bD$ has measure zero, then the operator \[B_{D,\mu}:L^2(D,\mu)\rightarrow \mathcal{O}(D)\cap L^2(D,\mu)\]
defined above is the \emph{Bergman projection} of the domain $D$ with respect to the weight $\mu$.

\subsection{Applications to Bergman and Szeg\"o projections}\label{bergman_sec}

We recall that a domain $D$ in a manifold $Y$ is said to be \emph{smoothly bounded} if it admits a smooth defining function, that is, a function $\rho\in C^\infty(V;\R)$, where $V$ is an open neighborhood of $bD$, such that $d\rho\neq 0$ on $bD$ and $D\cap V=\{\rho<0\}$. The boundary of a smoothly bounded domain is a real hypersurface of $Y$, and in particular it has measure zero (see the discussion after Setting B in Section \ref{settings_sec}). 

We have the following theorem. 

\begin{thm}\label{bergman_thm}
	Let $Y$ be a complex manifold equipped with a smooth positive measure $\mu$, and let $D\subseteq Y$ be a precompact and smoothly bounded domain.
	Assume that the algebra $A(D):=C(\overline{D})\cap \mathcal{O}(D)$ admits at least a boundary peak point $z_0\in bD$, that is, there exists $F\in A(D)$ such that $F(z_0)=1$ and $|F(z)|<1$ for every $z\in \overline{D}\setminus\{z_0\}$. Then the Bergman projection $B_{D,\mu}$ admits no bounded extension to $L^1(D,\mu)$. 
\end{thm}

We derive Theorem \ref{bergman_thm} from Theorem \ref{abstract_2_thm}, applied to \begin{eqnarray*}
	X=\overline D,\quad \mu, \quad H=\mathcal{O}(D)\cap L^2(D,\mu),\quad A=C(\overline{D})\cap \mathcal{O}(D),
\end{eqnarray*} 
as in Setting B. Assumption i') of Theorem \ref{abstract_2_thm} is contained in the statement, and assumption ii') is a consequence of the following lemma. 

\begin{lem}\label{noatom_lem} Let $Y$ be a complex manifold equipped with a smooth positive measure $\mu$, and let $D\subseteq Y$ be a precompact and smoothly bounded domain.
	Let $\{h_k\}_{k}\subset \mathcal{O}(D)\cap L^1(D,\mu)$, and assume that $\sigma=\lim_k h_k\mu\in \mathcal{M}(\overline{D})$ in the weak-$\star$ topology of $\mathcal{M}(\overline{D})$. Then $\sigma$ is nonatomic, that is, $\sigma\{z\}=0$ for every $z\in \overline{D}$. 
\end{lem}

\begin{proof}
	The hypothesis means that \[
	\int_{\overline{D}}\varphi d\sigma = \lim_k\int_D\varphi h_kd\mu\qquad \forall \varphi\in C(\overline{D}). 
	\]
	
	We show that $\sigma$ has no atoms on the boundary. The absence of atoms in the interior is easier to prove, or follows considering smaller domains $D'\subset D$.
	
	Fix $z_0\in bD$, an exhaustion $\{D_j\}_{j\in \N}$ of $D$, with $D_j$ smoothly bounded and compactly contained in $D$, and a smooth complex vector field $L\in\Gamma(Y,\C TY)$ supported in a neighborhood of $z_0$ and satisfying the following properties: \begin{itemize}
		\item[i)] $L$ does not vanish at $z_0$;
		\item[ii)] $L$ is of type $(0,1)$, that is, $L=\sum_{j=1}^na_j\frac{\partial}{\partial \overline{z}_j}$ in holomorphic local coordinates;
		\item[iii)] $L$ is tangent to $bD_j$ for every $j$ (and then necessarily to $bD$). 
	\end{itemize} 
	Denote by $L^\dagger$ the formal adjoint of $L$ with respect to $\mu$. If $\varphi\in C^\infty(Y)$, the conditions above and the assumption that $h_k$ is holomorphic give \[
	\int_{D_j}h_kL^\dagger\varphi d\mu=0 \qquad \forall j,k. 
	\]
	Letting first $j$, and then $k$, tend to $\infty$, we conclude that \begin{equation}\label{dagger_identity}
		\int_{\overline D}L^\dagger\varphi d\sigma=0\qquad \forall \varphi\in C^\infty(Y). 
	\end{equation}
	For the sake of clarity we now localize our analysis to a neighborhood $U$ of $z_0$ and choose local (nonholomorphic) coordinates $x_1, \ldots, x_{2n}$ on $U$ with respect to which $z_0=0$, $D\cap U=\{x_{2n}<0\}$, and $L^\dagger=\sum_{j=1}^{2n-1}b_j(x)\frac{\partial}{\partial x_j}+b(x)$. Notice that there is no $\frac{\partial}{\partial x_{2n}}$ term because of the tangency assumption on $L$. Writing $x=(x',x_{2n})\in \R^{2n-1}\times \R$, identity \eqref{dagger_identity} gives in particular \[
	\int_{\{x_{2n}\leq 0\}}\eta(R x_{2n}) L^\dagger(\psi(x'))d\sigma=0,
	\] for every $\psi\in C^\infty_c(\R^{2n-1})$, $\eta\in C^\infty_c(\R)$ with small support, and $R\geq 1$. If $\eta(0)=1$, letting $R$ tend to $+\infty$, we obtain \[
	\int_{\{x_{2n}\leq 0\}}1_{\{x_{2n}=0\}}L^\dagger(\psi(x'))d\sigma=\int_{\{x_{2n}= 0\}}L^\dagger(\psi(x'))d\sigma_b(x')=0,
	\]
	where $\sigma_b$ is the restriction of $\sigma$ to the boundary. In other words, $\sigma_b$, thought of as a distribution on $bD$, solves a first order PDE with principal part $\sum_{j=1}^{2n-1}b_j(x)\frac{\partial}{\partial x_j}$. Since $L$ does not vanish at $z_0$, we have $b_j(0)\neq0$ for some $j$, and $\sigma_b$ cannot have any atom at $z_0$ (as in the proof of Theorem \ref{settingA_thm}, this is a simple consequence of Lemma \ref{key_lem} below). \end{proof}

The proof of Theorem \ref{bergman_thm} is complete. We now show that it may be applied when $Y=\C^n$, as stated in Theorem \ref{intr_bergman_thm} of the Introduction.

\begin{proof}[Proof of Theorem \ref{intr_bergman_thm}]
	We show that the hypothesis of Theorem \ref{bergman_thm} is verified by any bounded domain $D\subset \C^n$, without any smoothness assumption (notice that smoothness is nevertheless needed for the proof of Lemma \ref{noatom_lem}). We use a standard observation: $D$ is contained in a Euclidean ball $B(z_1,R)$ such that $\overline D \cap bB(z_1,R)$ consists of a single point $z_0\in b D$. Then $F(z):=R^{-2}\left\langle z-z_1,z_0-z_1\right\rangle$ is the desired peaking function.
\end{proof}

The hypotheses of Theorem \ref{bergman_thm} may be weakened: all is needed is the smoothness of $bD$ in a neighborhood of the peak point $z_0$, because this is all the smoothness one needs to prove Lemma \ref{noatom_lem}. The reader may easily formulate a corresponding generalization of Theorem \ref{intr_bergman_thm}. 

We finally prove Theorem \ref{intr_szego_thm}. 

\begin{proof}[Proof of Theorem \ref{intr_szego_thm}]
We may apply the second part of Theorem \ref{settingA_thm}. The peak point is provided by the (standard) argument proving Theorem \ref{intr_bergman_thm} above (restrictions of holomorphic functions to $bD$ are CR functions). 
	\end{proof}

\section{Two lemmas about wave front sets}\label{wf_sec}

Here we discuss two lemmas about the wave front sets of certain measures that will be used in Section \ref{szego_sec}. Recall that if $f$ is any distribution on $\R^n$, then its wavefront set $\textrm{WF}(f)\subseteq T^*\R^n$ is defined as follows: $\textrm{WF}(f)=\bigcup_{x\in \R^n}\mathrm{WF}_x(f)$ and $\xi\notin \mathrm{WF}_x(f)$ if and only if there exists $\chi\in C^\infty_c(\R^n)$ such that $\chi(x)\neq0$ and the Fourier transform $\widehat{\chi f}(\xi)$ decays rapidly in a conical neighborhood of $\xi$. Here we are identifying $T_x^*\R^n$ with $\R^n$ as usual. The notion is local and invariant under diffeomorphisms, and thus it is possible to define the wave front set of a distribution on a manifold. We refer to \cite[Chapter VIII]{hormanderI} for the theory of wave front sets. 

The first lemma is a little variant of Theorem 8.1.5 of \cite{hormanderI}. We give a proof for the sake of completeness.

\begin{lem}\label{singular_measure_lem}
	If $\sigma$ is a regular complex Borel measure on a smooth manifold $Y$, and $M$ is a smooth embedded submanifold such that the restriction $\sigma_{|M}$ is not the zero measure, then \begin{equation}\label{wf_fact}
		\textrm{WF}(\sigma)\supset \{(x,\xi)\in T^*Y\colon \ x\in S \text{ and } 0\neq \xi\perp T_xM\},
	\end{equation}
	where $S$ is the support of $\sigma_{|M}$ and $\perp$ denotes orthogonality with respect to the pairing of co-vectors and vectors.
\end{lem}

\begin{proof}
	Since \eqref{wf_fact} is of local nature and it is invariant under diffeomorphisms, we can assume without loss of generality that $Y$ is a neighborhood of the origin in $\R^N$, that $M=\{x \in Y\colon x_{m+1}=\ldots=x_N=0\}$ for some $m<N$ (the case $\dim M=\dim Y$ is trivial), and that $x=0\in S$. We write the space variables as $x=(x',x'')\in \R^m\times \R^{N-m}$ and the cotangent variables as $\xi=(\xi',\xi'')\in \R^m\times \R^{N-m}$, identifying $T_x^*\R^N\equiv\R^N$ in the usual way. Since $0$ is in the support of $\sigma_{|M}$, for every small $\varepsilon>0$ there exists a test function $\varphi_\varepsilon(x')$ supported on $\{|x'|<\varepsilon\}$ such that $\varphi_\varepsilon(0)\neq0$ and $\int_M\varphi_\varepsilon d\sigma_{|M}\neq 0$. If $\xi=(0,\xi'')$ and we put $\psi_{\varepsilon, t}(x):=\varphi_\varepsilon(x')\eta(tx'')$, where $\eta(x'')$ is a test function supported on $\{|x''|<1\}$ such that $\eta(0)=1$ and $t$ is a large positive constant, we have \begin{eqnarray*}
		\left|\widehat{\psi_{\varepsilon, t}\sigma}(\xi)-\int_M\varphi_\varepsilon d\sigma_{|M}\right| &=& \left|\int_{Y\setminus M}e^{-i\xi\cdot x}\psi_{\varepsilon, t}d\sigma\right|\\
		&\leq& \int_{Y\setminus M}|\psi_{\varepsilon, t}|d|\sigma|\\
		&\leq& ||\varphi_\varepsilon||_\infty||\eta||_\infty\cdot |\sigma|\{x\in Y\colon 0<|x''|\leq t^{-1}\},
	\end{eqnarray*}
	where $|\sigma|$ is the total variation of $\sigma$. By the continuity from above of $|\sigma|$, we may find $t$ so large that the last quantity is strictly smaller than the modulus of $\int_M\varphi_\varepsilon d\sigma_{|M}$. By the arbitrariness of $\varepsilon$, we can thus ensure that $\psi_{\varepsilon, t}$ has arbitrarily small support near $0$ and that $\widehat{\psi_{\varepsilon, t}\sigma}(\xi)$ is not decaying as $\xi=(0,\xi'')$ goes to $\infty$ along any ray. This implies that $(0,\xi'')\in \textrm{WF}_0(\sigma)$ for every $\xi''\neq 0$, as we wanted. 
\end{proof}

Our second lemma on wave front sets deals with a more specialized setting. We recall that a submanifold $M$ of a CR manifold $X$ is \emph{characteristic} at $x\in M$ if $T_xM\subset H_xX=\Re(T_{1,0}X)_x$. 

\begin{lem}\label{key_lem}	
	Let $(X,T_{1,0}X)$ be a CR manifold of hypersurface type and let $\mu$ be a smooth positive measure. Assume that $\sigma$ is a CR regular complex Borel measure. Finally, assume that $M\subset X$ is a connected smooth embedded submanifold of positive codimension such that $\sigma_{|M}$ is not the zero measure. 
	
	Then $M$ has codimension one and it is characteristic at every point of the support of $\sigma_{|M}$. \end{lem}

\begin{proof}
	Since the conclusion is local, we may assume that $X$ is a neighborhood of $0\in \R^{2n+1}$, where the bundle $T_{0,1}X$ is generated by \[
	L_j:=\sum_{k=1}^{2n+1}a_{jk}(x)\partial_{x_k} \qquad (j=1,\ldots, n),
	\] for certain smooth complex-valued functions $a_{jk}(x)$, and $\mu$ may be identified with a smooth positive function on $X$. The assumption is that \[
	\int_XL_j^\dagger \varphi d\sigma = 0\qquad\forall \varphi\in C^\infty_c(X),
	\] where \[ L_j^\dagger=-\sum_{k=1}^{2n+1}a_{jk}\partial_{x_k}-\mu^{-1}\sum_{k=1}^{2n+1}\partial_{x_k}(\mu a_{jk})
	\] Thus, $\sigma$ satisfies a system of first order PDEs with principal parts given by the $L_j$'s. By a theorem of H\"ormander (Theorem 8.3.1 of \cite{hormanderI}), \[
	\textrm{WF}(\sigma)\subset\bigcap_j\textrm{Char}(L_j),
	\]
	where the characteristic sets are defined by \[\textrm{Char}(L_j)=\left\{(x,\xi)\in T^*X\colon\ \xi\neq 0,\quad \sum_{k=1}^{2n+1}a_{jk}(x)\xi_k=0\right\},\]
	and we are identifying $T_x^*X$ with $\R^{2n+1}$. Thus, \begin{equation}\label{inclusion_wf1}
		\textrm{WF}(\sigma)\subset\left\{(x,\xi)\in T^*X\colon \xi\neq 0,\quad \xi\perp HX\right\},
	\end{equation}
	because $HX$ is generated by \[
	\Re(L_j)=\sum_{k=1}^{2n+1}\Re(a_{jk})(x)\partial_{x_k},\quad  \Im(L_j)=\sum_{k=1}^{2n+1}\Im(a_{jk})(x)\partial_{x_k}\qquad (j=1,\ldots, n). 
	\]
	
	Using Lemma \ref{singular_measure_lem}, we get the inclusion \begin{equation}\label{inclusion_wf2}
		\textrm{WF}(\sigma)\supset \{(x,\xi)\in T^*X\colon \ x\in S  \text{ and } 0\neq \xi\perp T_xM\},
	\end{equation} where $S$ is the support of $\sigma_{|M}$. Combining \eqref{inclusion_wf1} and \eqref{inclusion_wf2}, we obtain that $H_xX\subset T_xM$ at every $x\in S$. Since $M$ has positive codimension, this forces $T_xM$ to be equal to $H_xX$ at every such point, proving that $M$ has codimension one and is characteristic at points of the support of $\sigma_{|M}$. \end{proof}

\section{Szeg\"o projections on real analytic CR manifolds}\label{szego_sec}

In this section we prove Theorem \ref{intr_CR_thm} of the Introduction. Let $X$ be a compact, connected and real analytic CR manifold of hypersurface type $(X,T_{1,0}X)$ equipped with a smooth positive measure $\mu$. We refer to Setting A of Section \ref{settings_sec}, where the Szeg\"o projection $\Pi:L^2(X,\mu)\rightarrow \mathrm{CR}^2(X,\mu)$ is defined. Recall that $\mathrm{CR}^2(X,\mu)$ is the space of $L^2$ CR functions on $X$. We need the following definition, where the measure $\mu$ plays no r\^{o}le. 

\begin{dfn}[Exceptional real analytic CR structures]\label{B_def}
We say that $(X,T_{1,0}X)$ is exceptional if the following property holds: a $\C$-valued real analytic function $f$ is CR if and only if both $\Re(f)$ and $\Im(f)$ are CR. 	
\end{dfn}

Notice that if both $\Re(f)$ and $\Im(f)$ are CR, then $f$ is CR. Hence, since $f$ is CR if and only if $if$ is CR, to verify that $(X,T_{1,0}X)$ is exceptional it is enough to check that if $f$ is real analytic and CR, then $\Re(f)$ is CR. 

The following is our main result about $L^1$ (un)boundedness of Szeg\"o projections.

\begin{thm}\label{CR_thm} Under the assumptions above, if the Szeg\"o projection $\Pi$ admits a bounded extension to $L^1(X,\mu)$, then $(X,T_{1,0}X)$ is exceptional. 	
\end{thm}

The next two sections contain an alternative characterization of the notion of exceptional CR structure and a proof of Theorem \ref{CR_thm}, respectively. Combining these two items, we get Theorem \ref{intr_CR_thm}. 

\subsection{An alternative characterization of exceptional CR structures}\label{exceptional_sec}

We recall that a CR manifold is said to be \emph{Levi flat} if the horizontal bundle $HX=\Re(T_{1,0}X)$ or, equivalently, the complex vector bundle $T_{1,0}X\oplus T_{0,1}X$, is involutive. In this case, $X$ is foliated by leaves carrying a natural complex structure, see \cite{barletta_dragomir} for a detailed discussion. This foliation is usually called the \emph{Levi foliation} of $X$. If $X$ is of hypersurface type, which is the case of interest to us, the leaves have real codimension one.

\begin{prop}\label{typeBcharacterization_prop}
	Let $(X,T_{1,0}X)$ be a compact, connected, real analytic CR manifold of hypersurface type. Then $(X,T_{1,0}X)$ is exceptional if and only if one of the following two conditions holds:
	\begin{enumerate}
		\item there are no nonconstant real analytic CR functions; 
		\item $(X,T_{1,0}X)$ is Levi flat and every leaf of the Levi foliation is compact. 
	\end{enumerate}
\end{prop}

As remarked above, Theorem \ref{CR_thm} and Proposition \ref{typeBcharacterization_prop} combined give Theorem \ref{intr_CR_thm}. 

\begin{proof}
Let us prove the \emph{only if} part first. Let $(X,T_{1,0}X)$ be as in the statement and assume that there exists a nonconstant real-analytic CR function $f$. Without loss of generality, we may assume that $u:=\Re(f)$ is nonconstant. Since $(X,T_{1,0}X)$ is exceptional, $u$ is CR. Given a local section $L$ of $T_{1,0}X$, we have $Lu=\overline{L}u=0$ (because $u$ is real-valued) and therefore $HX\subset \ker(du)$. Hence, on the open set $\Omega:=\{x\in X\colon du(x)\neq 0\}$ we have that $HX=\ker(du)$. Thus, $HX$ is involutive on $\Omega$. Since $\Omega$ is dense, because $X$ is connected and $u$ is real analytic, $X$ is Levi flat. 

We need to show that every leaf of the Levi foliation is closed. To this end, we employ a theorem of Inaba (Theorem 1 of \cite{inaba}) asserting that continuous CR functions on a compact Levi flat CR manifold are constant on the leaves of the Levi foliation. Notice that CR functions are holomorphic on the leaves (with respect to the induced complex structure), and hence necessarily constant on compact leaves: Inaba's theorem guarantees that the same is true on noncompact leaves, whose existence we have not ruled out yet. Thus, let us argue by contradiction and assume that there is $x_0\in X$ and a leaf $M$ such that $x_0\notin M$ and $x_0$ is the limit of a sequence $\{x_k\}_k$ of points of $M$. Let $(U,\Phi)$ be a foliated chart such that:\begin{enumerate} 
		\item $U$ is an open neighborhood of $x_0$;
		\item $\Phi:U\rightarrow B\times I$ is a real-analytic diffeomorphism such that $\Phi(x_0)=(0,0)$, $B$ is a connected open neighborhood of $0$ in $\R^{2n}$, and $I\subset\R$ is an open interval containing $0$;
		\item the horizontal bundle on $U$ is the pullback of the tangent bundle to the slices $M_t=\{(z,t)\colon\ z\in B\}$ ($t\in I$). 
	\end{enumerate}
	It is clear that $\Phi(M\cap U)$ does not intersect $M_0$, while the sequence $\{x_k\}_k$ gives us a sequence $t_k\in I\setminus \{0\}$ converging to $0$ such that $M_{t_k}\subset \Phi(M\cap U)$. By Inaba's Theorem, the CR function $f$ is constant on $M$ and, in particular, $f\circ\Phi^{-1}$ is constant on $\bigcup_kM_{t_k}$. By real analyticity, $f\circ\Phi^{-1}$ is identically constant on $B\times I$, and by connectedness, $f$ must be identically constant on $X$. This is the desired contradiction.  
	
The \emph{if} part is simpler. If (1) holds, then $(X,T_{1,0}X)$ is clearly exceptional, while if (2) holds, CR functions are constant on the leaves, and the same is true of their real parts, which are therefore CR too. 
\end{proof}

An inspection of the proof shows that if $(X,T_{1,0}X)$ is Levi flat, then it must be exceptional (a noncompact leaf forces every real analytic CR function to be constant).

\subsection{Proof of Theorem \ref{CR_thm}}

First of all, for every $x\in X$, Lemma \ref{representing_meas_lem} yields a measure $\sigma_x\in\mathcal{M}_0$ such that 
 \begin{equation}\label{representing}
 	\int_X gd\overline{\sigma_x} = g(x)\qquad \forall g \text{ continuous and CR}.
 \end{equation}

Recalling the remark after Definition \ref{B_def}, our task is to show that if $f$ is a real analytic CR function, then $u:=\Re(f)$ is CR. We may assume without loss of generality that $u$, and a fortiori $f$, is nonconstant. Then the set $I:=\{u(x)\colon x\in X\}$ is a compact nondegenerate interval, and by Proposition \ref{outer_prop} any $s\in I$ is the projection of at least a point $s+it(s)$ in the outer boundary of $K:=f(X)$. As in the proof of Theorem \ref{abstract_thm}, we invoke Corollary \ref{bishop_mergelyan_cor} to Bishop's Theorem and Mergelyan's Theorem to get a sequence of holomorphic polynomials $\{P_{s,k}\}_k$ and a function $F_s:K\rightarrow \C$ such that: \begin{enumerate}
	\item as $k$ tends to $\infty$, $P_{s,k}$ converges to $F_s$, uniformly on $K$;
	\item $F_s$ is a peaking function relative to the point $s+it(s)$, that is, $F_s(s+it(s))=1$ and $|F_s(z)|<1$ for every $z\in K\setminus\{s+it(s)\}$.
\end{enumerate} 
Since $(P_{s,k}\circ f)^N$ is CR for every $k,N\in \N$ and $s\in I$, we can use the representing property \eqref{representing} with $x\in f^{\leftarrow}\{s+it(s)\}$, and let first $k$ and then $N$ tend to $\infty$, to conclude that
\begin{eqnarray} 
	\sigma_x(f^{\leftarrow}(s+it(s)))=1\qquad \forall s\in I, \ x\in f^{\leftarrow}\{s+it(s)\}.\label{mass} 
\end{eqnarray}

At this point the argument has to depart from that of Theorem \ref{abstract_thm}, because the measures $\sigma_x$ are not absolutely continuous with respect to $\mu$ (and we do not want to exhibit any contradiction). 

First of all, we use the real analytic Sard's theorem, stating that if $v$ is a real-valued real analytic function defined on a compact real analytic manifold $Y$, then the level sets $\{y\in Y\colon u(y)=s\}$ are either empty or smooth real-analytic hypersurfaces, for any $s$ outside a finite set of real values. This form of Sard's theorem follows immediately from Theorem 1 of \cite{souvcek_souvcek}, which states that if $D\subset\R^n$ is open and $v:D\rightarrow\R$ is real analytic, then the set of critical values of $v_{|D'}$ is finite for any open set $D'$ compactly contained in $D$. In particular, there exists a set $E:=\{s_1,\ldots, s_m\}\subset I$ such that if $s\in I\setminus E$, then $V_s:=\{x\in X\colon u(x)=s\}$ is a compact smooth real analytic hypersurface. 

Fix $s\in I\setminus E$ and $x\in f^{\leftarrow}\{s+it(s)\}$. Since $f^{\leftarrow}(s+it(s))\subset V_s$, by \eqref{mass} we have the crucial information that \emph{the complex measure $\sigma\equiv\sigma_x$ has nontrivial restriction to the smooth hypersurface $V_s$}, because it assigns measure $1$ to a Borel subset of $V_s$. Let $M$ be any component of $V_s$ such that $\sigma$ restricted to $M$ is not the zero measure (at least a component has this property). We denote by $S$ the nonempty support of $\sigma_{|M}$ and by $C$ the set of characteristic points of $M$, namely those points $x\in M$ where $T_xM=H_xX$. Since $\mathcal{M}_0\subseteq \mathrm{Sol}(T_{0,1}X;\mathcal{M}(X))$, as remarked after introducing Setting A in Section \ref{settings_sec}, a first application of Lemma \ref{key_lem} gives the inclusion $S\subset C$. 

The set $C$ is a real analytic subset of $M$. In fact, if $L_j=X_j+iY_j$ ($j=1,\ldots, n$) is a local basis of $T_{1,0}X$, with $\{X_j, Y_j\}_j$ real-analytic real vector fields, then $C$ is locally defined by the system of equations \begin{equation}\label{characteristic_set}
u=s,\quad  X_ju=0,\quad Y_ju=0\qquad j=1,\ldots, n. 
\end{equation}
See remark (1) after the statement of Theorem \ref{intr_CR_thm} in the introduction. Thus, by Lojasiewicz's structure theorem for real analytic sets (see \cite[Section 6.3]{krantz_parks}), in a neighborhood of any $x\in S$ the set $C$ is the union of finitely many smooth real analytic submanifolds of $M$. At least one of them, which we denote by $C_0$, must have the property that $\sigma_{|C_0}$ is not the zero measure. Another application of Lemma \ref{key_lem} shows that $C_0$ must have codimension one in $X$, and hence be open in $M$. By analyticity, $C_0=C=M$. 

We proved that for any $s\in I\setminus E$, the hypersurface $V_s$ has at least one component $V_s^0$ such that \eqref{characteristic_set} holds at every point of $V_s^0$. Thus, the set defined by the conditions \[
\overline{L}u=0\qquad\forall L\in \Gamma(X;T_{1,0}X)
\] is a real analytic subset of $X$ that contains $\cup_{s\in I\setminus E}V_s^0$. By analyticity and connectedness, it must be equal to the whole manifold $X$. In other words, $u$ is a CR function. This completes the proof that $X$ is exceptional. 

\bibliographystyle{alpha} 
\bibliography{L1}
\end{document}